\documentclass[12pt]{amsart}
\usepackage{txfonts}      
\usepackage{amssymb}
\usepackage{eucal}
\usepackage{amsmath}
\usepackage{amscd}
\usepackage{xcolor}
\usepackage{multicol}
\usepackage[all]{xy}           
\usepackage{graphicx}
\usepackage{color}
\usepackage{colordvi}
\usepackage{xspace}
\usepackage{tikz}
\usepackage{makecell}
\usepackage{appendix}
\usepackage{enumerate,enumitem}
\usepackage{amsthm}
\usepackage[thicklines]{cancel}
\usepackage[compress]{cite}
\usepackage{ifpdf}
\ifpdf
\usepackage[colorlinks,final,backref=page,hyperindex]{hyperref}
\else
\usepackage[colorlinks,final,backref=page,hyperindex,hypertex]{hyperref}
\fi

\usepackage[active]{srcltx} 
\newtheorem{thm}{Theorem}[section]
\newtheorem{lem}[thm]{Lemma}
\newtheorem{cor}[thm]{Corollary}
\newtheorem{pro}[thm]{Proposition}
\theoremstyle{definition}
\newtheorem{defi}[thm]{Definition}

\newtheorem{rmk}[thm]{Remark}

\title[Quasi-triangular, factorizable anti-dendriform bialgebras and relative Rota-Baxter operators]
 {Quasi-triangular, factorizable anti-dendriform bialgebras and relative Rota-Baxter operators }

\author{Qinxiu Sun}
\address{Department of Mathematics, Zhejiang University of Science and Technology, Hangzhou, 310023} \email{qxsun@126.com}

\author{Min Wu}
\address{Department of Mathematics, Zhejiang University of Science and Technology, Hangzhou, 310023}
         \email{wuminzju@163.com}
         

\begin{document}
\begin{abstract}
We introduce the notion of quasi-triangular anti-dendriform bialgebras, which can be induced by
the solutions of the AD-YBE whose symmetric parts are invariant. A factorizable anti-dendriform bialgebra 
leads to a factorization
of the underlying anti-dendriform algebra. Moreover, relative Rota-Baxter operators with
 weights are introduced to characterize the solutions of the AD-YBE whose symmetric 
 parts are invariant. Finally, we interpret factorizable 
anti-dendriform bialgebras in terms of quadratic Rota-Baxter anti-dendriform algebras.

\end{abstract}

\maketitle

\vspace{-1.2cm}

\tableofcontents

\vspace{-1.2cm}

\allowdisplaybreaks

\section{Introduction}  
Anti-dendriform algebras were first studied by Gao, Liu and Bai in \cite{15}, which provide a decomposition of associativity.
However, unlike dendriform algebras where the left and right multiplication operators form the bimodules 
for the underlying associative algebra, 
it is the negative left and right multiplication operators that constitute the bimodules 
for the underlying sum-associative algebra in the anti-dendriform case.
 Moreover, anti-dendriform algebras exhibit a close relationship 
with anti-pre-Lie algebras. Investigations into anti-dendriform algebras and their 
connections to Novikov-type algebras establish a foundational framework for 
understanding anti-pre-Lie algebras \cite{15}. Explicitly, these relationships are captured by the following commutative diagram:
\begin{displaymath}
\xymatrix{
 \text{anti-dendriform algebras} \ar[d] \ar[r] & \text{anti-pre-Lie algebras}  \ar[d] \\
\text{associative algebras} \ar[r] & \text{Lie algebras.}}
\end{displaymath}

A bialgebra structure on an algebraic structure is defined by equipping it with a comultiplication
 that satisfies compatibility conditions with the multiplication. 
 Lie bialgebras \cite{8} introduced by Drinfeld in the early 1980s, which
 are closely related to the classical Yang-Baxter equation and play an important role 
 in the infinitesimal description of quantum groups. 
 Extending Drinfeld’s ideas, V. Zhelyabin introduced the notion of associative D-bialgebras \cite{32,33}.
  The associative analog of the Lie bialgebra
is antisymmetric infinitesimal bialgebras \cite{1}, which were interpreted in terms of 
 double constructions of Frobenius algebras \cite{4}.  
Analogous methods have been extended in recent years to construct bialgebra theories 
for numerous other algebraic structures, such as
pre-Lie algebras \cite{3}, dendriform algebras\cite{4}, 3-Lie algebras \cite{5}, Leibniz
algebras \cite{28}, Novikov algebra \cite{16}, perm algebras \cite{19}, 
pre-Novikov algebra \cite{22} and anti-pre-Lie algebras \cite{24}. 

Within the theory of Lie bialgebras, coboundary Lie bialgebras especially quasi-triangular Lie bialgebras play important roles
in mathematical physics. Factorizable Lie bialgebras \cite{022} constitute a specialized subclass of 
quasi-triangular Lie bialgebras that bridge classical $r$-matrices with certain factorization problems. 
 These structures find diverse applications in integrable systems \cite{102,103}. 
 Recently, these foundational results have been extended to 
 antisymmetric infinitesimal bialgebras \cite{029}, pre-Lie bialgebras \cite{033} and Leibniz bialgebras \cite{07}.

We studied the bialgebra theory for
anti-dendriform algebras in \cite{030}, characterized by double constructions of associative algebras via commutative Connes cocycles 
and certain matched pairs of 
associative algebras as well as anti-dendriform algebras. 
While a skew-symmetric solution of the anti-dendriform Yang-Baxter
equation (AD-YBE) yields a coboundary anti-dendriform bialgebra \cite{030}, this paper investigates
how solutions without skew-symmetry induce anti-dendriform bialgebras. We verify that the solutions of the AD-YBE
whose symmetric parts are invariant gives rise to quasi-triangular anti-dendriform bialgebras. Furthermore, 
Factorizable anti-dendriform bialgebras as a special class of quasi-triangular anti-dendriform bialgebras 
 are considered. The double space of any anti-dendriform bialgebra has a factorizable anti-dendriform bialgebra
structure. We demonstrate that the solutions of the AD-YBE
whose symmetric parts are invariant in terms of relative Rota-Baxter operators on anti-dendriform algebras.

The paper is organized as follows. In Section 2, we recall some known facts about anti-dendriform 
algebras and anti-dendriform bialgebras. In Section 3, we introduce the notion of quasi-triangular 
anti-dendriform bialgebras, which can be induced by solutions of the AD-YBE whose symmetric parts are invariant. 
In Section 4, we investigate factorizable anti-dendriform bialgebras, a special class of quasi-triangular anti-dendriform bialgebras.
 A factorizable anti-dendriform bialgebra admits a factorization of the underlying anti-dendriform algebra, 
 and the double of an anti-dendriform bialgebra naturally forms a factorizable anti-dendriform bialgebra. 
 In Section 5, we introduce relative Rota-Baxter operators with weights to interpret solutions of 
 the AD-YBE whose symmetric parts are invariant. Moreover, 
 we show that there exists a one-to-one correspondence between 
 factorizable anti-dendriform bialgebras and quadratic Rota-Baxter anti-dendriform algebras.

Throughout the paper, $k$ is a field.  All vector spaces and algebras are over $k$. 
 All algebras are finite-dimensional, although many results still hold in the infinite-dimensional case.

\section{Anti-dendriform algebras and bialgebras}

Let us begin to recall some basic knowledge on anti-dendriform algebras and bialgebras.

An {\bf anti-dendriform algebra}  is a vector space $A$ together with two 
bilinear maps $\succ,\prec:A\otimes A\longrightarrow A$ such that
 \begin{align}&\label{A1}x\succ (y\succ z)=-(x\cdot y)\succ z=-x\prec (y\cdot z)=(x\prec y)\prec z,\\&
 \label{A2}(x\succ y)\prec z=x\succ (y\prec z),\end{align} 
  for all $x,y,z\in A$, where $x\cdot y= x\succ y+x\prec y$. Note that $(A,\cdot)$ is 
  an associative algebra, which is called the \textbf{associated associative 
  algebra} of the anti-dendriform algebra $(A,\succ,\prec)$.
  While $(A,\succ,\prec)$ is called the \textbf{compatible 
  anti-dendriform algebra} structure on $(A,\cdot)$.
   Further details on anti-dendriform algebras can be found in \cite{15}.
 
\begin{defi} \cite{030} A {\bf representation (bimodule)} of an anti-dendriform algebra $(A,\succ,\prec)$ is a quintuple
  $(V,l_{\succ}, r_{\succ},l_{\prec},r_{\prec})$, 
 where $V$ is a
vector space and $l_{\succ}, r_{\succ},l_{\prec},r_{\prec} : A \longrightarrow \hbox{End}
(V)$ are linear maps satisfying the following conditions:
\begin{align}&\label{R1}l_{\succ}(x)l_{\succ}(y)=-l_{\succ}(x\cdot y)=-l_{\prec}(x)l_{\cdot}(y)=l_{\prec}(x\prec y),\\&
\label{R2}r_{\succ}(x\succ y)=-r_{\succ}(y)r_{\cdot}(x)=-r_{\prec}(x\cdot y)=r_{\prec}(y)r_{\prec}(x),\\&
\label{R3}l_{\succ}(x)r_{\succ}(y)=-r_{\succ}(y)l_{\cdot}(x)=-l_{\prec}(x)r_{\cdot}(y)=r_{\prec}(y)l_{\prec} (x),\\&
\label{R4}l_{\prec}(x \succ y)=l_{\succ}(x)l_{\prec}(y),\ \ \
r_{\prec}(y)r_{\succ}(x)=r_{\succ}(x\prec y),\ \ \
r_{\prec}(y)l_{\succ}(x)=l_{\succ}(x)r_{\prec} (y),\end{align}
where $l_{\cdot}(x)=l_{\succ}(x)+l_{\prec}(x),~r_{\cdot}(x)=r_{\succ}(x)+r_{\prec}(x)$ and $x\cdot y=x\succ y+x\prec y$ for all
$x,y\in A$.
\end{defi}

Let $A$ and $V$ be vector spaces. For
a linear map $f: A \longrightarrow \hbox{End} (V)$, define a linear
map $f^{*}: A \longrightarrow \hbox{End} (V^{*})$ by $\langle
f^{*}(x)u^{*},v\rangle=\langle u^{*},f(x)v\rangle$ for all $x\in A,
u^{*}\in V^{*}, v\in V$, where $\langle \ , \ \rangle$ is the usual
pairing between $V$ and $V^{*}$.

\begin{pro} \cite{030} \label{zr} Let $(A,\succ,\prec)$ be an anti-dendriform algebra
and $(A,\cdot)$ be its associated associative algebra. Assume that
$(V,l_{\succ}, r_{\succ},l_{\prec},r_{\prec})$ is a
representation of $(A,\succ,\prec)$. Then
\begin{enumerate}
	\item $(V,-l_{\succ},-r_{\prec})$ is a representation of $(A,\cdot)$.
 \item $(V,l_{\succ}+l_{\prec},r_{\succ}+r_{\prec})$ is a representation of $(A,\cdot)$.
	\item $(V^{*},-(r_{\prec}^{*}+r_{\succ}^{*}),l_{\prec}^{*},r_{\succ}^{*},-(l_{\prec}^{*}+l_{\succ}^{*}))$ is also a
	representation of $(A,\succ,\prec)$. We call it the {\bf dual representation}.
	\item $(V^{*},-r_{\prec}^{*},-l_{\succ}^{*})$ is a
	representation of $(A,\cdot)$.
\item $(V^{*},-(r_{\prec}^{*}+r_{\succ}^{*}),-(l_{\prec}^{*}+l_{\succ}^{*}))$ is a
	representation of $(A,\cdot)$. 
\end{enumerate}
\end{pro}

\begin{defi} \cite{030} \label{DB1} An anti-dendriform bialgebra is a quintuple $(A,\succ,\prec,\Delta_{\succ},\Delta_{\prec})$
such that $(A,\succ,\prec)$ is an anti-dendriform algebra, $(A,\Delta_{\succ},\Delta_{\prec})$ is 
an anti-dendriform coalgebra and the following compatible conditions hold:
\begin{equation*}\Delta_{\prec}(x\cdot y)=(R_{\cdot}(y)\otimes I)\Delta_{\prec}(x)-(I\otimes L_{\succ} (x))\Delta_{\prec}(y),\end{equation*}
\begin{equation*}\Delta_{\succ}(x\cdot y)=(I\otimes L_{\cdot}(x))\Delta_{\succ}(y)-(R_{\prec} (y)\otimes I)\Delta_{\succ}(x),\end{equation*}
\begin{equation*}\Delta(x\prec y)=(R_{\prec} (y)\otimes I)\Delta(x)-(I\otimes L_{\prec} (x))\Delta_{\succ}(y),\end{equation*}
\begin{equation*}\Delta(x\succ y)=(I\otimes L_{\succ} (x))\Delta(y)-(R_{\succ} (y)\otimes I)\Delta_{\prec}(x),\end{equation*}
\begin{equation*} (L_{\succ} (x)\otimes I)\Delta(y)+(R_{\succ} (y)\otimes I)\tau\Delta_{\succ}(x)
-(I\otimes L_{\prec} (y))\tau\Delta_{\prec}(x)-(I\otimes R_{\prec} (x))\Delta(y)=0,\end{equation*}
\begin{equation*}(I\otimes R_{\cdot}(y))\Delta_{\succ}(x)+(L_{\succ} (y)\otimes I)\Delta_{\succ}(x)
-\tau(I\otimes R_{\prec} (x))\Delta_{\prec}(y)-\tau
(L_{\cdot}(x)\otimes I)\Delta_{\prec}(y)=0,\end{equation*}
where $\cdot=\succ+\prec,~\Delta=\Delta_{\succ}+\Delta_{\prec}$ and $R_{\cdot}=R_{\prec}+R_{\succ},~L_{\cdot}=L_{\prec}+L_{\succ}$.
\end{defi}

\begin{rmk} $(A,\Delta_{\succ},\Delta_{\prec})$ is 
an anti-dendriform coalgebra 
 if and only if $(A^{*},\succ_{A^{*}},\prec_{A^{*}})$
 is an anti-dendriform algebra, where $\succ_{A^{*}},\prec_{A^{*}}$ are the linear dual of $\Delta_{\succ},\Delta_{\prec}$
 respectively, that is,
 \begin{align}&\label{Dc1}\langle \Delta_{\succ}(x),\zeta\otimes \eta\rangle=\langle x,\zeta\succ_{A^{*}} \eta\rangle
\\&\label{Dc2}\langle \Delta_{\prec}(x),\zeta\otimes \eta\rangle=\langle x,\zeta\prec_{A^{*}} \eta\rangle,~\forall~x\in A,\zeta,\eta\in A^{*}.
\end{align}
 Therefore, an anti-dendriform bialgebra $(A,\succ,\prec,\Delta_{\succ},\Delta_{\prec})$ is sometimes
denoted by $(A,\succ,\prec,A^{*},$ \ \ \ $\succ_{A^{*}},\prec_{A^{*}})$, where the anti-dendriform
 algebra structure $(A^{*},\succ_{A^{*}},\prec_{A^{*}})$
 on the dual space $A^{*}$
corresponds to the anti-dendriform coalgebra $(A,\Delta_{\succ},\Delta_{\prec})$
through Eqs.~(\ref{Dc1})-(\ref{Dc2}).
\end{rmk}

Let $(A, \succ_{A},\prec_{A},\Delta_{\prec}, \Delta_{\succ})$ be
 an anti-dendriform
bialgebra. Then $(D=A\oplus A^{*},\succ_{D},\prec_{D})$
 is an anti-dendriform algebra, where
 \begin{align}\label{Db1}(x+a)\succ_{D}(y+b)=&x\succ_A y-(R_{\prec_{A^{*}}}^{*}+R_{\succ_{A^{*}}}^{*})(a)y
 +L_{\prec_{A^{*}}}^{*}(b)x
 \\&+a\succ_{A^{*}}b-(R_{\prec_A}^{*}+R_{\succ_A}^{*})(x)b+L_{\prec_A}^{*}(y)a\nonumber
,\end{align}
\begin{align}\label{Db2}
(x+a)\prec_{D}(y+b)=&x\prec_A y+R_{\succ_{A^*}}^{*}
(a)y-(L_{\prec_{A^{*}}}^{*}+L_{\succ_{A^{*}}}^{*}(b)x\\&+a\prec_{A^{*}}b+R_{\succ_A}^{*}
(x)b-(L_{\prec_A}^{*}+L_{\succ_A}^{*})(y)a\nonumber,
\end{align}
for all $x,y\in A,a,b\in A^{*}$.
$(D=A\oplus A^{*},\succ_{D},\prec_{D})$ is called the double anti-dendriform algebra. 

Let $W$ be a vector space with a binary operation $\ast$. Suppose that
 $r=\sum\limits_{i}a_i\otimes b_i \in W\otimes W$. Put
\begin{eqnarray*}
r_{12}\ast r_{13}:=\sum_{i,j}a_i\ast a_j\otimes b_i\otimes b_j,\;r_{23}\ast r_{12}:=\sum_{i,j}a_j\otimes a_i\ast b_j\otimes b_i,\;
r_{31}\ast r_{23}:=\sum_{i,j}b_i\otimes a_j\otimes a_i\ast b_j,\\
r_{21}\ast r_{13}:=\sum_{i,j}b_i\ast a_j\otimes a_i\otimes b_j,\;
r_{32}\ast r_{21}:=\sum_{i,j}b_j\otimes b_i\ast a_j\otimes a_i,\;
r_{31}\ast r_{32}:=\sum_{i,j}b_i\otimes b_j\otimes a_i\ast a_j,\\
r_{13}\ast r_{32}:=\sum_{i,j}a_i\otimes b_j\otimes b_i\ast a_j,\;
r_{23}\ast r_{21}:=\sum_{i,j}b_j\otimes a_i\ast a_j\otimes b_i,\;
r_{21}\ast r_{31}:=\sum_{i,j}b_i\ast b_j\otimes a_i\otimes a_j,\\
r_{23}\ast r_{13}:=\sum_{i,j}a_i \otimes a_j \otimes b_i\ast b_j,\;
r_{12}\ast r_{31}:=\sum_{i,j}a_i\ast b_j\otimes b_i\otimes a_j.
\end{eqnarray*}

\begin{thm} \cite{030} \label{YE3} Let $(A,\succ,\prec)$ be an anti-dendriform algebra and 
$r=\sum_{i}a_i\otimes b_i\in A\otimes A$. Assume that
$\Delta_{\succ,r},\Delta_{\prec,r}$ defined by 
\begin{align}&\label{CD1}
\Delta_{\succ,r}(x)=-(R_{\prec}(x)\otimes I+I\otimes L_{\cdot}(x))r,
\\&\label{CD2}\Delta_{\prec,r}(x)=(R_{\cdot}(x)\otimes I+I\otimes L_{\succ}(x))r.
\end{align}
Then $(A,\succ,\prec,\Delta_{\succ,r},\Delta_{\prec,r})$ is an anti-dendriform bialgebra 
if and only if the following equations hold:
\begin{align}\label{CD3}[L_{\succ}( x)R_{\succ}(y)\otimes I-R_{\succ}(y)\otimes R_{\prec}( x)+I\otimes R_{\prec}(x)L_{\prec}(y)
-L_{\succ}(x)\otimes L_{\prec}( y)](r+\tau (r))=0,\end{align}
\begin{equation}\label{CD4}(R_{\prec}(x)\otimes I+I\otimes L_{\cdot}(x))(L_{\succ}(y)\otimes I+I\otimes R_{\cdot}(y))(r+\tau (r))=0,\end{equation}
\begin{align}\label{CD5}(R_{\prec}(x)\otimes I \otimes I-I\otimes I\otimes L_{\succ}(x))(r_{12}\prec r_{13}+r_{23}\cdot r_{12}+r_{13}\succ r_{23})
=0,\end{align}
\begin{align}\label{CD6}
(I\otimes I\otimes L_{\cdot}(x)+R_{\prec}(x)\otimes I \otimes I)(r_{13}\cdot r_{23}-r_{12}\succ r_{13}+r_{23}\prec r_{12})=0,\end{align}
\begin{align}\label{CD7}
(R_{\cdot}(x)\otimes I \otimes I+(I\otimes I\otimes L_{\succ}(x))(r_{12}\cdot r_{13}+
r_{23}\succ r_{12}-r_{13}\prec r_{23})=0,\end{align}
\begin{align}\label{CD8}&
 (R_{\prec}(x)\otimes I \otimes I)(r_{23}\prec r_{12}+r_{13}\cdot r_{23}-r_{12}\succ r_{13})
\\=&(I\otimes I\otimes L_{\succ}(x))(r_{12}\cdot r_{13}+r_{23}\succ r_{12}-r_{13}\prec r_{23})\nonumber
.\end{align}
\end{thm}

\begin{defi} \cite{030} Let $(A,\succ,\prec)$ be an anti-dendriform algebra and 
$r\in A\otimes A$. The equation
\begin{equation} \label{YE} r_{12}\cdot r_{13}+r_{23}\succ r_{12}-r_{13}\prec r_{23}=0.\end{equation}
 is called the {\bf anti-dendriform Yang-Baxter equation} in $(A,\succ,\prec)$  or {\bf AD-YBE} in short. 
\end{defi}

\begin{thm} \cite{030} \label{BY} Let $r\in A\otimes A$ be skew-symmetric.
Assume that
$\Delta_{\succ,r},\Delta_{\prec,r}$ defined by Eqs.~\eqref{CD1}-\eqref{CD2}.
Then $(A,\succ,\prec,\Delta_{\succ,r},\Delta_{\prec,r})$ is an anti-dendriform bialgebra if
$r$ is a solution of the AD-YBE in $(A,\succ,\prec)$.
\end{thm}

Additional details regarding bialgebras are available in \cite{030}.
\section{Quasi-triangular anti-dendriform bialgebras }
A skew-symmetric solution of the AD-YBE generates 
an (coboundary) anti-dendriform bialgebra. In this section, we examine how solutions of the AD-YBE 
produce anti-dendriform bialgebras without imposing skew-symmetry constraints. 

Let $A$ be a vector space and $r\in A\otimes A$. Define $T_{r}:A^{*}\longrightarrow A$ by
 \begin{equation}\label{To}
\langle T_{r}(\zeta),\eta\rangle=\langle r,\zeta\otimes \eta\rangle,~~\forall~
\zeta,\eta\in A^{*}.
\end{equation}
It is apparent that $T_{r}^{*}=T_{\tau(r)},~T_{r+\tau(r)}^{*}=T_{r+\tau(r)}$.

\begin{defi} \label{In}
 Let $(A,\succ,\prec)$ be an anti-dendriform algebra and $r\in A\otimes A$. We say that $r$ is {\bf invariant} if 
 \begin{align}&\label{IE3}
(R_{\prec}(x)\otimes I+I\otimes L_{\cdot}(x))r=0,
\\&\label{IE4}(R_{\cdot}(x)\otimes I+I\otimes L_{\succ}(x))r=0,~\forall~x\in A.
\end{align}
 
\end{defi}
 
 \begin{lem} \label{In1}
 Let $(A,\succ,\prec)$ be an anti-dendriform algebra and $r\in A\otimes A$. Then 
$r$ is invariant if and only if the following equations hold:
  \begin{align}&\label{IE5} R_{\cdot}^{*}(T_{r}(\zeta))\eta=-L_{\prec}^{*}(T_{\tau(r)}(\eta) \zeta,
\\&\label{IE6}R_{\succ}^{*}(T_{r}(\zeta))\eta=-L_{\cdot}^{*}(T_{\tau(r)}(\eta) \zeta,~~\forall~\zeta,\eta\in A^{*}.\end{align}
Moreover, Eqs.~(\ref{IE5})-(\ref{IE6}) are equivalent to the following equations: 
 \begin{align}&\label{IE7}
L_{\cdot}(x) T_{r}(\zeta)=-T_{r}(R_{\prec}^{*}(x)\zeta),\ \ \
 L_{\succ}(x) T_{r}(\zeta) =-T_{r}(R_{\cdot}^{*}(x)\zeta),~~\forall~x\in A,\zeta\in A^{*}.
\end{align}

\end{lem}
\begin{proof} 
For all $~x\in A,\zeta,\eta\in A^{*}$, we have
\begin{align*}
\langle \zeta\otimes\eta,(I\otimes L_{\cdot}(x)+R_{\prec}(x)\otimes I)r
\rangle
&=\langle \zeta\otimes L_{\cdot}^{*}(x)\eta,r\rangle
+\langle \eta\otimes R_{\prec}^{*}(x)\zeta,\tau(r)\rangle
\\&=\langle  T_{r}(\zeta),L_{\cdot}^{*}(x)\eta\rangle
+\langle T_{\tau(r)}(\eta), R_{\prec}^{*}(x)\zeta\rangle
\\&=\langle  x\cdot T_{r}(\zeta),\eta\rangle
+\langle T_{\tau(r)}(\eta)\prec x, \zeta\rangle
\\&=\langle  x,R_{\cdot}^{*}(T_{r}(\zeta))\eta\rangle
+\langle x, L_{\prec}^{*}(T_{\tau(r)}(\eta)) \zeta\rangle,
\end{align*}
Thus, Eq.~(\ref{IE3}) $\Longleftrightarrow$ Eq.~(\ref{IE5}). Similarly,
 Eq.~(\ref{IE4}) $\Longleftrightarrow$ Eq.~(\ref{IE6}), Eqs.~(\ref{IE5})-(\ref{IE6}) 
 hold if and only if Eq.~(\ref{IE7}) holds. 
\end{proof}

\begin{pro} \label{Si}
 Let $(A,\succ,\prec)$ be an anti-dendriform algebra and $r\in A\otimes A$. Then the following
 conditions are equivalent:
  \begin{enumerate}
\item $r+\tau(r)$ is invariant.
\item The following equations hold:
\begin{align}&\label{IE8}
L_{\cdot}(x) T_{r+\tau(r)}(\zeta)=-T_{r+\tau(r)}(R_{\prec}^{*}(x)\zeta),\ \ \ 
L_{\succ}(x) T_{r+\tau(r)}(\zeta) =-T_{r+\tau(r)}(R_{\cdot}^{*}(x)\zeta),~~\forall~x\in A,\zeta\in A^{*}.\end{align}
\item The following equations hold:
\begin{align}&\label{IE9}
T_{r+\tau(r)}(L_{\cdot}^{*}(x)\zeta) =-R_{\prec}(x)T_{r+\tau(r)}(\zeta),\ \ \ 
 T_{r+\tau(r)}L_{\succ}^{*}(x)(\zeta) =-R_{\cdot}(x)T_{r+\tau(r)}(\zeta),~~\forall~x\in A,\zeta\in A^{*}.
\end{align}
\item The following equations hold:
\begin{align}&\label{IE10} R_{\cdot}^{*}(T_{r+\tau(r)}(\zeta))\eta=-L_{\prec}^{*}(T_{r+\tau(r)}(\eta) \zeta, \ \ \
R_{\succ}^{*}(T_{r+\tau(r)}(\zeta))\eta=-L_{\cdot}^{*}(T_{r+\tau(r)}(\eta) \zeta,~~\forall~\zeta,\eta\in A^{*}.\end{align}
\end{enumerate}
 \end{pro}
  
 \begin{proof} In view of Lemma \ref{In1}, $r+\tau(r)$ is invariant if and only if Eq.~(\ref{IE8}) holds.
 Since $T_{r+\tau(r)}=T_{r+\tau(r)}^{*}$, Eq.~(\ref{IE8}) holds if and only if Eq.~ (\ref{IE9}) holds.
 By Eq.~(\ref{IE5})-(\ref{IE6}), $r+\tau(r)$ is invariant if and only if Eq.~(\ref{IE10}) holds. The proof is completed.
\end{proof}
 Combining Eqs.~(\ref{IE8})-(\ref{IE10}), we obtain for all~$x\in A,~\zeta,\eta\in A^{*}$,
 \begin{align}&\label{IE11}
L_{\prec}(x) T_{r+\tau(r)}(\zeta)=T_{r+\tau(r)}(R_{\succ}^{*}(x)\zeta), 
\\&\label{IE12}T_{r+\tau(r)}(L_{\prec}^{*}(x)\zeta) =R_{\succ}(x)T_{r+\tau(r)}(\zeta),
\\&\label{IE13} R_{\prec}^{*}(T_{r+\tau(r)}(\zeta))\eta=L_{\succ}^{*}(T_{r+\tau(r)}(\eta) \zeta.\end{align}
 
\begin{pro} \label{Da1}
Let $(A,\succ,\prec,\Delta_{\succ,r},\Delta_{\prec,r})$ be a anti-dendriform bialgebra and $r\in A\otimes A$, 
where the comultiplications
$\Delta_{\succ,r},\Delta_{\prec,r}$ are defined by Eqs.~(\ref{CD1})-(\ref{CD2}).
 Then the anti-dendriform algebra structure $(\succ_r,\prec_r)$ on $A^{*}$ is
given by 
\begin{align}&\label{IE1}\zeta\succ_{r}\eta=-R_{\cdot}^{*}(T_{r}(\zeta))\eta-L_{\prec}^{*}(T_{\tau(r)}(\eta))\zeta,
\\&\label{IE2}\zeta\prec_{r}\eta=
R_{\succ}^{*}(T_{r}(\zeta))\eta+L_{\cdot}^{*}(T_{\tau(r)}(\eta))\zeta.
\end{align}
And the associative algebra structure on  $\cdot_r$ on $A^{*}$ is
given by 
\begin{align}\label{IE0}\zeta\cdot_{r}\eta=L_{\succ}^{*}(T_{\tau(r)}(\eta))\zeta-R_{\prec}^{*}(T_{r}(\zeta))\eta,
\end{align}
where $\cdot=\succ+\prec,~
 L_{\cdot}=L_{\succ}+L_{\prec},~R_{\cdot}=R_{\succ}+R_{\prec}$.
\end{pro}
\begin{proof} For all $\zeta,\eta\in A^{*}$ and $x\in A$,
\begin{align*}\langle \zeta\succ_{r}\eta,x\rangle
&=\langle \zeta\otimes\eta,\Delta_{\succ,r}(x)\rangle
=-\langle \zeta\otimes\eta,(I\otimes L_{\cdot}(x)+R_{\prec}(x)\otimes I)r
\rangle
\\&=-\langle \zeta\otimes L_{\cdot}^{*}(x)\eta,r\rangle
-\langle \eta\otimes R_{\prec}^{*}(x)\zeta,\tau(r)\rangle
=-\langle T_{r}(\zeta),L_{\cdot}^{*}(x)\eta\rangle
-\langle T_{\tau(r)}(\eta),R_{\prec}^{*}(x)\zeta\rangle
\\&=-\langle x\cdot T_{r}(\zeta),\eta\rangle-
\langle T_{\tau(r)}(\eta)\prec x,\zeta\rangle
=-\langle R_{\cdot}^{*}(T_{r}(\zeta))\eta+L_{\prec}^{*}(T_{\tau(r)}(\eta))\zeta
,x\rangle,
\end{align*}
which indicates that Eq. (\ref{IE1}) holds.
By the same token, Eq. (\ref{IE2}) holds. Using Eqs. (\ref{IE1})-(\ref{IE2}), we get Eq. (\ref{IE0}).
\end{proof}

\begin{thm} Let $(A,\succ,\prec)$ be an anti-dendriform algebra and 
$r=\sum_{i}a_i\otimes b_i\in A\otimes A$. Assume that
$\Delta_{\succ,r},\Delta_{\prec,r}$ given by Eqs.~(\ref{CD1})-(\ref{CD2}). If $r$ is a solution of the 
AD-YBE in $(A,\succ,\prec)$ and $r+\tau(r)$ is invariant.
Then $(A,\succ,\prec,\Delta_{\succ,r},\Delta_{\prec,r})$ is an anti-dendriform bialgebra.
\end{thm}
\begin{proof} Given the invariance of $r+\tau(r)$ is invariant, 
Eqs.~(\ref{IE8})-(\ref{IE13}) hold and
\begin{align*}
(R_{\prec}(x)\otimes I+I\otimes L_{\cdot}(x)) (r+\tau(r))=0, \ \ \ (R_{\cdot}(x)\otimes I+I\otimes L_{\succ}(x)) (r+\tau(r))=0,~\forall~x\in A.
\end{align*}
Thus Eq.~(\ref{CD4}) holds and
\begin{align*}&\langle \zeta\otimes\eta, [L_{\succ}( x)R_{\succ}(y)\otimes I-R_{\succ}(y)\otimes R_{\prec}( x)+I\otimes R_{\prec}(x)L_{\prec}(y)-L_{\succ}(x)\otimes L_{\prec}( y)](r+\tau (r))\rangle
\\=&\langle R_{\succ}^{*}(y)L_{\succ}^{*}( x)\zeta\otimes\eta - R_{\succ}^{*}(y)\zeta\otimes R_{\prec}^{*}( x)\eta
+\zeta\otimes L_{\prec}^{*}(y)R_{\prec}^{*}(x)\eta-L_{\succ}^{*}(x)\zeta\otimes L_{\prec}^{*}( y)\eta, 
r+\tau (r)\rangle
\\=&\langle T_{r+\tau (r)}(R_{\succ}^{*}(y)L_{\succ}^{*}( x)\zeta),\eta\rangle-
\langle T_{r+\tau (r)}(R_{\succ}^{*}(y)\zeta),R_{\prec}^{*}( x)\eta\rangle+
\langle T_{r+\tau (r)}(\zeta),L_{\prec}^{*}(y)R_{\prec}^{*}(x)\eta\rangle
\\&-\langle T_{r+\tau (r)}(L_{\succ}^{*}(x)\zeta),L_{\prec}^{*}(y)\eta\rangle
\\=&\langle - L_{\prec}(y)R_{\cdot}(x)T_{r+\tau (r)}(\zeta)-
 R_{\prec}(x)L_{\prec}(y)T_{r+\tau (r)}(\zeta)+  R_{\prec}(x)L_{\prec}(y)T_{r+\tau (r)}(\zeta)
+ L_{\prec}(y)R_{\cdot}(x)T_{r+\tau (r)}(\zeta),\eta  \rangle
\\=&0,\end{align*}
which implies that Eq.~(\ref{CD3}) holds.
Note that
\begin{align*}&(R_{\prec}(x)\otimes I \otimes I-I\otimes I\otimes L_{\succ}(x))(r_{12}\prec r_{13}+r_{23}\cdot r_{12}+r_{13}\succ r_{23})
\\=&(R_{\prec}(x)\otimes I \otimes I-I\otimes I\otimes L_{\succ}(x))[-\sigma_{123}D(r)
+r_{23}\cdot (r_{12}+r_{21})\\&+(r_{13}+(r_{12}+r_{21})\prec r_{13}+r_{31})\succ r_{23}-r_{21}\prec (r_{13}+r_{31})]
\\=&(R_{\prec}(x)\otimes I \otimes I-I\otimes I\otimes L_{\succ}(x))[
\sum_{i,j}(I\otimes L_{\cdot}(a_i)+R_{\prec}(a_i)\otimes I)(r+\tau(r)) \otimes b_i
\\&+(I\otimes I\otimes R_{\succ}(b_i)-L_{\prec}(b_i)\otimes I\otimes I)(\tau\otimes I)[a_i\otimes (r+\tau(r))]
\\=&0,\end{align*}
which indicates that Eq.~(\ref{CD5}) holds. Analogously, we can prove that Eqs.~(\ref{CD6})-(\ref{CD8}) hold.
Combining Theorem \ref{YE3}, we get the statement.
 \end{proof}

\begin{defi} \label{Qt1}
 Let $(A,\succ,\prec)$ be an anti-dendriform algebra and $r\in A\otimes A$. If $r$ is a solution of the AD-YBE and 
 $r+\tau(r)$ is invariant, then the 
anti-dendriform bialgebra $(A, \succ,\prec,\Delta_{\succ,r},\Delta_{\prec,r})$ induced by $r$ is called
  a {\bf quasi-triangular} anti-dendriform bialgebra. 
 In particular, if $r$ is skew-symmetric, $(A, \succ,\prec,\Delta_{\succ,r},\Delta_{\prec,r})$
is called a {\bf triangular} anti-dendriform bialgebra.
\end{defi}

\begin{pro} \label{Ya} Let $(A,\succ,\prec)$ be an anti-dendriform algebra and
$r\in A\otimes A$. Then  \begin{enumerate}
\item
 $r$ is a solution of the
AD-YBE $ D(r)=r_{12}\cdot r_{13}+r_{23}\succ r_{12}-r_{13}\prec r_{23}=0$ in $(A,\succ,\prec)$
if and only if the following equation holds:
 \begin{equation*}T_{\tau(r)}(\eta)\cdot T_{\tau(r)}(\zeta)=T_{\tau(r)}(R_{\prec}^{*}(T_{r}(\eta))\zeta
-L_{\succ}^{*}(T_{\tau(r)}(\zeta))\eta),~~\forall~\eta,\zeta\in A^{*}.
\end{equation*}
\item
 $r$ is a solution of the
AD-YBE $ D(r)=r_{12}\cdot r_{13}+r_{23}\succ r_{12}-r_{13}\prec r_{23}=0$ 
in $(A,\succ,\prec)$
 if and only if the following equation holds:
 \begin{equation*}T_{r}(\eta)\prec T_{r}(\zeta)=T_{r}(R_{\succ}^{*}(T_{r}(\eta))\zeta+
L_{\cdot}^{*}(T_{\tau(r)}(\zeta))\eta),~~\forall~\eta,\zeta\in A^{*}.
\end{equation*}
\item
 $r$ is a solution of the equation $ D_{1}(r)=r_{23}\cdot r_{12}+r_{13}\succ r_{23}+r_{12}\prec r_{13}=0$ 
 in $(A,\succ,\prec)$
 if and only if the following equation holds:
 \begin{equation*}T_{r}(\eta)\succ T_{r}(\zeta)=-T_{r}(L_{\prec}^{*}(T_{\tau(r)}(\zeta))\eta
+R_{\cdot}^{*}(T_{r}(\eta))\zeta),~~\forall~\eta,\zeta\in A^{*}.
\end{equation*}
\item
 $r$ is a solution of the equation $ D_{1}(r)=r_{23}\cdot r_{12}+r_{13}\succ r_{23}+r_{12}\prec r_{13}=0$
 in $(A,\succ,\prec)$
if and only if the following equation holds:
 \begin{equation*}T_{\tau(r)}(\eta)\prec T_{\tau(r)}(\zeta)=-T_{\tau(r)}(L_{\cdot}^{*}(T_{\tau(r)}(\zeta))\eta
+R_{\succ}^{*}(T_{r}(\eta))\zeta),~~\forall~\eta,\zeta\in A^{*}.
\end{equation*}
\item
 $r$ is a solution of the equation $ D_{2}(r)=r_{13}\cdot r_{23}-r_{12}\succ r_{13}+r_{23}\prec r_{12}=0$
 in $(A,\succ,\prec)$
if and only if the following equation holds:
 \begin{equation*}T_{r}(\eta)\cdot T_{r}(\zeta)=T_{r}(L_{\succ}^{*}(T_{\tau(r)}(\zeta))\eta
-R_{\prec}^{*}(T_{r}(\eta))\zeta),~~\forall~\eta,\zeta\in A^{*}.
\end{equation*}
\item
 $r$ is a solution of the equation $ D_{2}(r)=r_{13}\cdot r_{23}-r_{12}\succ r_{13}+r_{23}\prec r_{12}=0$
 in $(A,\succ,\prec)$
if and only if the following equation holds:
 \begin{equation*}T_{\tau(r)}(\eta)\succ T_{\tau(r)}(\zeta)=T_{\tau(r)}(L_{\prec}^{*}(T_{\tau(r)}(\zeta))\eta
+R_{\cdot}^{*}(T_{r}(\eta))\zeta),~~\forall~\eta,\zeta\in A^{*}.
\end{equation*}

 \end{enumerate}

\end{pro}

\begin{proof} According to Eq.~\eqref{To},  we have for all $\zeta,\eta,\theta\in A^{*}$,
\begin{align*}\langle \zeta\otimes \eta\otimes \theta,r_{12}\cdot r_{13}\rangle
&=\sum_{i,j}\langle \zeta\otimes \eta\otimes \theta,a_i\cdot a_j\otimes b_i\otimes b_j\rangle
\\&=\sum_{i,j}\langle \zeta,a_i\cdot a_j\rangle\langle \eta,b_i\rangle \langle \theta,b_j\rangle 
=\langle \zeta,T_{\tau(r)}(\eta)\cdot T_{\tau(r)}(\theta) \rangle
,\end{align*}
\begin{align*}\langle \zeta\otimes \eta\otimes \theta,r_{13}\prec r_{23}\rangle
&=\sum_{i,j}\langle \zeta\otimes \eta\otimes \theta,a_i \otimes a_j\otimes (b_i\prec b_j)\rangle
\\&=\sum_{i,j}\langle \zeta,a_i\rangle\langle \eta,a_j\rangle \langle \theta,b_i\prec b_j\rangle 
=\langle  \theta,T_{r}(\zeta) \prec T_{r}(\eta)\rangle
\\&=\langle  R_{\prec}^{*}(T_{r}(\eta))\theta,T_{r}(\zeta) \rangle
=\langle \zeta, T_{\tau(r)}(R_{\prec}^{*}(T_{r}(\eta))\theta) \rangle,\end{align*}
 \begin{align*}\langle \zeta\otimes \eta\otimes \theta,r_{23}\succ r_{12}
  \rangle
 &=\sum_{i,j}\langle \zeta\otimes \eta\otimes \theta,a_j \otimes (a_i\succ b_j)\otimes b_i\rangle
 \\& =\sum_{i,j}\langle \zeta,a_j\rangle\langle \eta,a_i\succ b_j\rangle \langle \theta,b_i\rangle 
 =\langle \eta,T_{\tau(r)}(\theta)\succ T_{r}(\zeta) \rangle
  \\&=\langle L_{\succ}^{*}(T_{\tau(r)}(\theta))\eta,T_{r}(\zeta)\rangle
 =\langle \zeta,T_{\tau(r)}(L_{\succ}^{*}(T_{\tau(r)}(\theta))\eta)\rangle.\end{align*}
 Thus, $r$ is a solution of the AD-YBE in $(A,\succ,\prec)$ is equivalent to that
\begin{equation*}T_{\tau(r)}(\eta)\cdot T_{\tau(r)}(\theta)=T_{\tau(r)}(R_{\prec}^{*}(T_{r}(\eta))\theta-L_{\succ}^{*}(T_{\tau(r)}(\theta))\eta).
\end{equation*}Similarly, Items (b)-(f) hold.

\end{proof}

\begin{thm}\label{Ya1} Let $(A,\succ,\prec)$ be an anti-dendriform algebra and
$r\in A\otimes A$. Assume that $r+\tau(r)$ is invariant. Then the following conditions hold:
\begin{enumerate}
\item $r$ is a solution of the AD-YBE $D(r)=0$ if and only if 
$r$ is a solution of the equations $D_1(r)=D_2(r)=0$.
\item $r$ is a solution of $D(r)=0$ if and only if 
$r$ is a solution of the equation $D(\tau(r))=r_{21}\cdot r_{31}+r_{32}\succ r_{21}-r_{31}\prec r_{32}=0$.
\end{enumerate}
\end{thm} 

\begin{proof} 
If $r$ is a solution of $D_1(r)=D_2(r)=0$, then
$r$ is a solution of $D(r)=D_1(r)-D_2(r)=0$. On the other hand, assume that $r$ is a solution of $D(r)=0$. 
In view of Eq.~(\ref{IE12}), for all $\zeta,\eta,\theta\in A^{*}$ we have
\begin{align*} &\langle \zeta\otimes\eta \otimes\theta,\sum_{i}(I\otimes I\otimes R_{\succ}(b_i)-
L_{\prec}(b_i)\otimes I\otimes I)
(\tau\otimes I)[a_i\otimes (r+\tau(r))]\rangle
\\=&\sum_{i}\langle \eta\otimes\zeta\otimes R_{\succ}(b_i)^{*}\theta -\eta\otimes L_{\prec}(b_i)^{*}\zeta\otimes\theta,
a_i\otimes (r+\tau(r)) \rangle
\\=&\langle \zeta\otimes  R_{\succ}^{*}(T_{r}(\eta))\theta-L_{\prec}^{*}(T_{r}(\eta))\zeta\otimes\theta,r+\tau(r) \rangle
 \\=&\langle T_{r+\tau(r)}(\zeta),  R_{\succ}^{*}(T_{r}(\eta))\theta\rangle
 -\langle T_{r+\tau(r)}(L_{\prec}^{*}(T_{r}(\eta))\zeta),\theta\rangle
\\=&\langle R_{\succ}(T_{r}(\eta)) T_{r+\tau(r)}(\zeta)-T_{r+\tau(r)}(L_{\prec}^{*}(T_{r}(\eta))\zeta),  \theta\rangle
\\=&0.
\end{align*} 
Thus,
\begin{align*} D_1(r)=&-\sigma_{123}D(r)+r_{23}\cdot (r_{12}+r_{21})+(r_{13}+r_{31})\succ r_{23}-
r_{21}\prec (r_{31}+r_{13})+(r_{12}-r_{21})\prec r_{13}
\\=&-\sigma_{123}D(r)+\sum_{i}(I\otimes L_{\cdot}(a_i)+R_{\prec}(a_i)\otimes I)(r+\tau(r))\otimes b_i\\&+(I\otimes I\otimes R_{\succ}(b_i)-
L_{\prec}(b_i)\otimes I\otimes I)
(\tau\otimes I)[a_i\otimes (r+\tau(r))]\\=&0.
\end{align*} 
 It follows that $D_2(r)=D_1(r)-D(r)=0.$ In all, Item (a) holds. Similarly, Item (b) holds.
The proof is completed.
\end{proof}

\begin{thm}\label{QB0} Let $(A,\succ,\prec)$ be an anti-dendriform algebra and
$r\in A\otimes A$. Suppose that $r+\tau(r)$ is invariant. Then the following conditions are equivalent:
 \begin{enumerate}
\item $r$ is
a solution of the
AD-YBE in $(A,\succ,\prec)$.
 \item $(A^{*},\cdot_r)$ is an associative algebra and the linear maps
$T_{r},-T_{\tau(r)}$ are both associative algebra
 homomorphisms from $(A^{*},\cdot_r)$ to $(A,\cdot)$.
 \item $(A^{*},\succ_r,\prec_r)$ is an anti-dendriform algebra and the linear maps
$T_{r},-T_{\tau(r)}$ are both anti-dendriform algebra
 homomorphisms from $(A^{*},\succ_r,\prec_r)$ to $(A,\succ,\prec)$.
 \end{enumerate}
\end{thm}
\begin{proof} 
Combining Proposition \ref{Ya}, Theorem \ref{Ya1} and Proposition \ref{Da1}, we deduce
 the conclusion.
\end{proof}

\begin{cor} \label{QB} Let $(A,\succ,\prec,\Delta_{\succ,r},\Delta_{\prec,r})$ be a quasi-triangular anti-dendriform bialgebra.
 Then  \begin{enumerate}
\item $T_{r},-T_{\tau(r)}$ are both associative algebra
 homomorphisms from $(A^{*},\cdot_r)$ to $(A,\cdot)$.
 \item $T_{r},-T_{\tau(r)}$ are anti-dendriform algebra
 homomorphisms from $(A^{*},\succ_r,\prec_r)$ to $(A,\succ,\prec)$.
 \end{enumerate}
\end{cor}

\begin{proof} It follows directly from Theorem \ref{QB0} and Definition \ref{Qt1}.
\end{proof}

\section{Factorizable anti-dendriform bialgebras}
In this section, we introduce the notion of factorizable 
anti-dendriform bialgebras, which is a special class of quasi-triangular anti-dendriform bialgebras.
We show that the double of an anti-dendriform bialgebra
is naturally a factorizable anti-dendriform bialgebra.

\begin{defi} \label{Qt} A quasi-triangular anti-dendriform bialgebra $(A, \succ,\prec, \Delta_{\succ,r}, \Delta_{\prec,r})$ is
called factorizable if  
  $T_{r+\tau(r)}:A^{*}\longrightarrow A$ given by Eq. (\ref{To}) is a linear isomorphism of vector spaces 
where $\Delta_{\succ,r}, \Delta_{\prec,r}$
are defined by Eqs.~(\ref{CD1})-(\ref{CD2}).\end{defi}

\begin{pro} \label{Qf} Assume that $(A, \succ,\prec, \Delta_{\succ,r}, \Delta_{\prec,r})$ is a quasi-triangular
 (factorizable) anti-dendriform bialgebra. Then $(A, \succ,\prec, \Delta_{\succ,\tau(r)}, \Delta_{\prec,\tau(r)})$
 is also a quasi-triangular
 (factorizable) anti-dendriform bialgebra.\end{pro}
 \begin{proof} This conclusion follows from Item (b) of Theorem \ref{Ya1} and Definition \ref{Qt}.
 \end{proof}
 
\begin{pro}
Let $(A, \succ,\prec, \Delta_{\succ,r}, \Delta_{\prec,r})$ be a factorizable anti-dendriform bialgebra.
Then $\mathrm{Im}(T_{r}\oplus T_{\tau(r)})$ is an anti-dendriform 
 subalgebra of the direct sum anti-dendriform  
algebra $A\oplus A$,
which is isomorphic to the anti-dendriform algebra $(A^{*},\succ_{r},\prec_{r})$. Furthermore, any $x\in A$ has a unique
decomposition $x=x_1-x_2$, where $(x_1,x_2)\in \mathrm{Im}(T_{r}\oplus T_{\tau(r)})$ and
\begin{equation*}T_{r}\oplus T_{\tau(r)}:A^{*}\longrightarrow A\oplus A, \ \ 
(T_{r}\oplus T_{\tau(r)})(\zeta)=(T_{r}(\zeta), -T_{\tau(r)}(\zeta)),~\forall~\zeta\in A^{*}. \end{equation*}
\end{pro}

\begin{proof} In view of Corollary \ref{QB}, $T_{r},-T_{\tau(r)}$ are anti-dendriform
 algebra 
homomorphisms and 
$T_{r+\tau(r)}$ is a linear isomorphism. It follows that
$\mathrm{Im}(T_{r}\oplus T_{\tau(r)})$ is an anti-dendriform subalgebra of the direct sum anti-dendriform
algebra $A\oplus A$ and $\mathrm{Im}(T_{r}\oplus T_{\tau(r)})\simeq A^{*}$ as anti-dendriform
algebras. Due to $T_{r+\tau(r)}$ being an isomorphism,
we get 
\begin{equation*}T_{r}T_{r+\tau(r)}^{-1}(x)+T_{\tau(r)}T_{r+\tau(r)}^{-1}(x)=T_{r+\tau(r)}T_{r+\tau(r)}^{-1}(x)=x,\end{equation*}
that is, $x=x_1-x_2$ with $x_1=T_{r}T_{r+\tau(r)}^{-1}(x)$ and $x_2=-T_{\tau(r)}T_{r+\tau(r)}^{-1}(x)$. 
The proof is finished.
\end{proof}
\begin{pro}
 Let $(A, \succ_{A},\prec_{A}, \Delta_{\succ,r}, \Delta_{\prec,r})$ 
 be a factorizable anti-dendriform bialgebra. Then the double anti-dendriform 
algebra $(D=A\oplus A^{*}, \succ_D,\prec_D)$
is isomorphic to the direct sum $A\oplus A$ of anti-dendriform algebras.
\end{pro}

\begin{proof} By Definition \ref{Qt1} and Definition \ref{Qt}, 
$T_{r+\tau(r)}$ is a linear isomorphism and $r+\tau(r)$ is invariant.
Define $\varphi:A\oplus A^{*}\longrightarrow A\oplus A$ by
\begin{equation}\label{FD1}\varphi(x,\zeta)=(x+T_{r}(\zeta),x-T_{\tau(r)}(\zeta)),~\forall~(x,\zeta)\in A\oplus A^{*}.\end{equation}
It is obvious that $\varphi$ is a bijection.
 By Eq.~(\ref{IE2}), we get
\begin{align*}&\langle\eta,
T_{r}(-(R_{\prec_A}^{*}+R_{\succ_A}^{*})(x)\zeta)+L_{\prec_{A^*}}^{*}(\zeta)x\rangle
\\=& -\langle T_{\tau(r)}(\eta), (R_{\prec_A}^{*}+R_{\succ_A}^{*})(x)\zeta
\rangle+\langle \zeta\prec_r\eta,x\rangle
\\=&-\langle \zeta,T_{\tau(r)}(\eta)\cdot_A x\rangle +
\langle L_{\cdot_A}^{*}(T_{\tau(r)}(\eta))\zeta+
R_{\succ}^{*}(T_{r}(\zeta))\eta,x\rangle
\\=&-\langle \zeta,T_{\tau(r)}(\eta)\cdot_A x\rangle +
\langle \zeta,T_{\tau(r)}(\eta)\cdot_A x\rangle+\langle \eta
, x\succ_A T_{r}(\zeta)\rangle
\\=&\langle \eta, x\succ_A T_{r}(\zeta)\rangle, 
\end{align*}
which yields that 
\begin{equation}\label{FD2}T_{r}(-(R_{\prec_A}^{*}+R_{\succ_A}^{*})(x)\zeta)+L_{\prec_{A^*}}^{*}(\zeta)x
=x\succ_A T_{r}(\zeta).\end{equation}
Analogously,  
\begin{equation}\label{FD3}T_{\tau(r)}((R_{\prec_A}^{*}+R_{\succ_A}^{*})(x)\zeta)+L_{\prec_{A^*}}^{*}(\zeta)x
=-x\succ_A T_{\tau(r)}(\zeta).\end{equation}
According to Eqs. (\ref{Db1}) and (\ref{FD1})-(\ref{FD3}),
\begin{align*}&\varphi(x\succ_D \zeta)=\varphi(L_{\prec_{A^*}}^{*}(\zeta)x
-(R_{\prec_A}^{*}+R_{\succ_A}^{*})(x)\zeta)\\
=&(T_{r}(-(R_{\prec_A}^{*}+R_{\succ_A}^{*})(x)\zeta)+L_{\prec_{A^*}}^{*}(\zeta)x,
T_{\tau(r)}((R_{\prec_A}^{*}+R_{\succ_A}^{*})(x)\zeta+L_{\prec_{A^*}}^{*}(\zeta)x)
\\=&(x\succ_A T_{r}(\zeta),-x\succ_A T_{\tau(r)}(\zeta))
=(x,x)\succ (T_{r}(\zeta),-T_{\tau(r)}(\zeta))
\\=&\varphi(x)\succ_A \varphi(\zeta).
\end{align*}
By the same token, $\varphi(\zeta\succ_D x)=\varphi(\zeta)\succ_A \varphi(x) $. Thus,
$\varphi((x,\zeta)\succ_D (y,\eta))=\varphi(x,\zeta)\succ \varphi(y,\eta) $ for all
$(x,\zeta), (y,\eta)\in A\oplus A^{*}$.
Taking the same procedure, we can prove that
$\varphi((x,\zeta)\prec_D (y,\eta))=\varphi(x,\zeta)\prec \varphi(y,\eta) $.
Thus, $\varphi$ is an anti-dendriform algebra isomorphism.
The proof is completed.
\end{proof}

\begin{thm} \label{Ft}
Assume that $(A, \succ_A,\prec_A, \Delta_{\succ,r}, \Delta_{\prec,r})$ is an anti-dendriform
bialgebra and $(D=A\oplus A^{*},\succ_{D},\prec_{D})$ is the double anti-dendriform algebra given in Section 2. 
 Assume that $\{e_1,...,e_n\}$ is a basis of $A$ and $\{e^{*}_1,...,e^{*}_n\}$ is the dual basis.
 Let \begin{equation*}r=\sum_{i=1}^{n}e_{i}\otimes e_{i}^{*}\in A\otimes A^{*}\subseteq D\otimes D.\end{equation*}
Then $(D,\succ_D,\prec_D, \Delta_{\succ,r},\Delta_{\prec,r})$ with 
$\Delta_{\succ,r},\Delta_{\prec,r}$ given by Eqs.~(\ref{CD1})-(\ref{CD2})
 is a factorizable anti-dendriform
bialgebra.

\end{thm}
\begin{proof}
Firstly, we prove that $r+\tau(r)=\sum_{i,j=1}^{n}(e_{i}\otimes e_{i}^{*}+e_{j}^{*}\otimes e_{j})$ is invariant. 
By Eqs.~(\ref{Db1}) and (\ref{Db2}), we have for all $x\in A$ 
\begin{align*}&(R_{\cdot_D}(x)\otimes I+I\otimes L_{\succ_D}(x))\sum_{i,j=1}^{n}(e_{i}\otimes e_{i}^{*}+e_{j}^{*}\otimes e_{j})
\\=&\sum_{i,j=1}^{n}e_{i}\cdot_{D} x\otimes e_{i}^{*}+e_{j}^{*}\cdot_{D}x\otimes e_{j}
+e_{i}\otimes x\succ_{D}e_{i}^{*}+e_{j}^{*}\otimes x\succ_De_{j}
\\=&\sum_{i,j=1}^{n}e_{i}\cdot_A x\otimes e_{i}^{*}-(R_{\prec_{A^*}}^{*}(e_{j}^{*})x+L_{\succ_A}^{*}(x)e_{j}^{*})
\otimes e_{j}
-e_{i}\otimes(R_{\cdot_A}^{*}(x)e_{i}^{*}-L_{\prec_{A^*}}^{*}(e_{i}^{*})x)
+e_{j}^{*}\otimes x\succ_A e_{j}.
\end{align*}
Note that
\begin{align*}& \sum_{i}^{n}e_{i}\cdot_A x\otimes e_{i}^{*}-e_{i}\otimes(R_{\cdot_A}^{*}(x)e_{i}^{*})=0,
\\&\sum_{i,j=1}^{n} e_{i}\otimes L_{\prec_{A^*}}^{*}(e_{i}^{*})x-(R_{\prec_{A^*}}^{*}(e_{j}^{*})x)
\otimes e_{j}=0,
\\&\sum_{j=1}^{n}e_{j}^{*}\otimes x\succ_A e_{j}-L_{\succ_A}^{*}(x)e_{j}^{*}\otimes e_{j}=0.
\end{align*}
Thus, $(R_{\cdot_D}(x)\otimes I+I\otimes L_{\succ_D}(x))(r+\tau(r))=0.$
Similarly, $(I\otimes L_{\cdot}(x)+R_{\prec}(x)\otimes I)(r+\tau(r))=0.$
By duality, we get
\begin{equation*}(R_{\cdot_D}(\zeta)\otimes I+I\otimes L_{\succ_D}(\zeta))(r+\tau(r))=0,
\ \ \
(I\otimes L_{\cdot}(\zeta)+R_{\prec}(\zeta)\otimes I)(r+\tau(r))=0\end{equation*}
for all $\zeta\in A^{*}$. Thus, according to Definition \ref{In}, $r+\tau(r)$ is invariant.
Secondly, we prove that $r$ is a solution of the AD-YBE in $(D,\succ_D,\prec_D)$.
 From Eqs.~(\ref{Db1}) and (\ref{Db2}), we get
\begin{align*}& r_{12}\cdot_D r_{13}+r_{23}\succ_D r_{12}-r_{13}\prec_D r_{23}
\\=&\sum_{i,j=1}^{n}e_{i}\cdot_D e_{j}\otimes e_{i}^{*}\otimes e_{j}^{*}+e_{j}\otimes e_{i}\succ_D e_{j}^{*} \otimes e_{i}^{*}
-e_{i}\otimes e_j \otimes e_{i}^{*}\prec_D e_{j}^{*}
\\=&\sum_{i,j=1}^{n}e_{i}\cdot_A e_{j}\otimes e_{i}^{*}\otimes e_{j}^{*}+e_{j}\otimes (L_{\prec_{A^*}}^{*}(e_{j}^{*})e_{i}-R_{\cdot_A}^{*}(e_{i})e_{j}^{*})
 \otimes e_{i}^{*}-e_{i}\otimes e_j \otimes e_{i}^{*}\prec_{A^*} e_{j}^{*}
=0.\end{align*}
In all, $(D,\succ_D,\prec_D, \Delta_{\succ,r},\Delta_{\prec,r})$ is a quasi-triangular anti-dendriform
bialgebra.
Furthermore, the linear maps $T_{r},T_{\tau(r)}:D^{*}\longrightarrow D$ are respectively defined by 
$T_{r}(\zeta,x)=\zeta,~T_{\tau(r)}(\zeta,x)=-x$ for all $x\in A,\zeta\in A^{*}$, 
which yields that
$T_{r+\tau(r)}(\zeta,x)=(\zeta,-x)$ is a linear isomorphism. Therefore, $(D,\succ_D,\prec_D, \Delta_{\succ,r},\Delta_{\prec,r})$
is a factorizable anti-dendriform
bialgebra.

\end{proof}
\section{Relative Rota-Baxter operators and the AD-YBE}
In this section, we introduce the notion of relative Rota-Baxter operators with 
weights on anti-dendriform algebras, 
which can demonstrates the solutions of the AD-YBE whose symmetric parts are invariant.

\begin{defi} Let $(A,\succ,\prec)$ and $(V,\succ_V,\prec_V)$ be anti-dendriform algebras. Assume that
 $(V,l_{\succ},r_{\succ},$\ \ \ \ $l_{\prec},r_{\prec})$ is a representation of $(A,\succ,\prec)$
 and the following conditions are satisfied:
 \begin{align*}&
 (l_{\succ}(x)a)\prec_V b=l_{\succ}(x)(a\prec_V b),\ \ \ (r_{\succ}(x)a)\prec_V b=a\succ_V(l_{\prec}(x)b)
 ,\ \ \ r_{\prec}(x)(a\succ_V b)=a\succ_V(r_{\prec}(x)b),\\&
 l_{\succ}(x)(a\prec_V b)=-(l_{\cdot}(x)a)\succ_V b=-l_{\prec}(x)(a\cdot_V b)=(l_{\prec}(x)a)\prec_V b, 
 \\&
 a\succ_V(l_{\succ}(x) b)=-(r_{\cdot}(x)a)\succ_V b=-a\prec_V (l_{\cdot}(x)b)
 =(r_{\prec}(x)a)\prec_V b, 
 \\&
 a\succ_V(r_{\succ}(x) b)=-r_{\succ}(x)(a\cdot_V b)=-a\prec_V (r_{\cdot}(x)b)
 =r_{\prec}(x)(a\prec_V b), 
\end{align*}
for all $x\in A$ and $a,b\in V$. Then $(V,\succ_V,\prec_V,l_{\succ},r_{\succ},l_{\prec},r_{\prec})$ is called
  an $A$-anti-dendriform algebra, where $\cdot=\succ+\prec,~\cdot_V=\succ_V+\prec_V$.
\end{defi}
 It is easy to check that $(V,\succ_V,\prec_V,l_{\succ},r_{\succ},l_{\prec},r_{\prec})$ is an $A$-anti-dendriform algebra
 if and only if $(A\oplus V,\succeq,\preceq)$ is an anti-dendriform algebra, where
 \begin{align*}&
 (x+a)\succeq (y+b)=x\succ y+l_{\succ}(x)b+r_{\succ}(y)a+a\succ_V b,\\&
 (x+a)\preceq (y+b)=x\prec y+l_{\prec}(x)b+r_{\prec}(y)a+a\prec_V b.
  \end{align*}

\begin{defi} Let $(A,\cdot)$ and $(V,\circ)$ be associative algebras. Assume that
 $(V,l_{A},r_{A})$ is a representation of $(A,\cdot)$
 and the following conditions are satisfied:
 \begin{align*}
 (l_{A}(x)a)\circ b=l_{A}(x)(a\circ b),\ \ \ a\circ (r_{A}(x)b)=r_{A}(x)(a\circ b)
 ,\ \ \ (r_{A}(x)a)\circ b=a\circ(l_{A}(x)b),
\end{align*}
for all $x\in A$ and $a,b\in V$. Then $(V,\circ,l_{A},r_{A})$ is called
  an $A$-associative algebra.
\end{defi}

 \begin{pro}
Let $(A,\succ,\prec)$ be an anti-dendriform algebra and $r\in A\otimes A$ be skew-symmetric and invariant. 
Define multiplications $\succeq_r,\preceq_r:A^{*}\otimes A^{*}\longrightarrow A^{*}$ by
\begin{align}&\label{AD1} \zeta\succeq_r\eta=R_{\cdot}^{*}(T_{r}(\zeta))\eta=L_{\prec}^{*}(T_{r}(\eta))\zeta, 
 \\&\label{AD2}\zeta\preceq_r\eta=-
R_{\succ}^{*}(T_{r}(\zeta))\eta=-L_{\cdot}^{*}(T_{r}(\eta))\zeta,~~\forall~\zeta,\eta\in A^{*}.
\end{align}
Then $(A^{*},\succeq_r,\preceq_r,-R_{\cdot}^{*},L_{\prec}^{*},R_{\succ}^{*},-L_{\cdot}^{*})$ is an $A$-anti-dendriform algebra
and $(A^{*},\cdot_r,-R_{\prec}^{*},-L_{\succ}^{*})$ is an $A$-associative algebra, where
\begin{equation}\label{AD3} \zeta\cdot_r\eta=R_{\prec}^{*}(T_{r}(\zeta))\eta=-L_{\succ}^{*}(T_{r}(\eta))\zeta,~~\forall~\zeta,\eta\in A^{*}.
\end{equation}
\end{pro}

\begin{proof} Using Eqs.~(\ref{R1})-(\ref{R4}), (\ref{IE8})-(\ref{IE13}) 
and (\ref{AD1})-(\ref{AD2}), 
we obtain the conclusion via direct computations.
\end{proof}

\begin{defi} Let $(A,\succ,\prec)$ be an anti-dendriform algebra and 
$(V,\succ_V,\prec_V,l_{\succ},r_{\succ},l_{\prec},r_{\prec})$ be an A-anti-dendriform algebra.
A relative Rota-Baxter operator $T$ of weight $\lambda$ on $(A,\succ,\prec)$ associated to
 $(V,\succ_V,\prec_V,l_{\succ},r_{\succ},l_{\prec},r_{\prec})$
   is a linear map $T:V\longrightarrow A$ satisfying
\begin{align*}&T(u)\succ T(v)=T (l_{\succ}(T(u))v+r_{\succ}(T(v))u+\lambda u\succ_V v),\\& 
T(u)\prec T(v)=T (l_{\prec}(T(u))v+r_{\prec}(T(v))u+\lambda u\prec_V v),~~\forall~u,v\in V.\end{align*}
When $u\succ_Vv=u\prec_Vv=0$ for all $u,v\in V$, then $T$ is simply $\mathcal O$-operator ( a relative Rota-Baxter operator) on
$(A,\succ,\prec)$ associated to a representation $(V,l_{\succ},r_{\succ},l_{\prec},r_{\prec})$.
In particular, a relative Rota-Baxter operator $P$ on $(A,\succ,\prec )$ associated to
 $(A,\succ,\prec,L_{\succ},R_{\succ},L_{\prec},R_{\prec})$ is called a Rota-Baxter operator of weight $\lambda$, 
 that is, $P:A\longrightarrow A$
is a linear map satisfying 
\begin{align*}&
P(x)\succ P(y)=P(P(x)\succ y)+x\succ P(y)+\lambda x\succ y),\\
 & P(x)\prec P(y)=P(P(x)\prec y)+x\prec P(y)+\lambda x\prec y) \end{align*}
for all $x, y, z \in A$.
$(A,\succ,\prec,P)$ is called a Rota-Baxter anti-dendriform algebra of weight $\lambda$.\end{defi}

Let $(A,\succ,\prec,P)$ be a Rota-Baxter anti-dendriform algebra of weight $\lambda$. 
It is easy to check that
  $(A,\succ,\prec,\tilde{P}=-\lambda I -P)$ is a Rota-Baxter anti-dendriform algebra of weight $\lambda$.

\begin{thm}
Let $(A,\succ,\prec)$ be an anti-dendriform algebra and $r\in A\otimes A$. Assume that $r+\tau(r)$ is invariant.
Then the following
conditions are equivalent.
 \begin{enumerate}
\item $r$ is a solution of the AD-YBE in $(A,\succ,\prec)$
 such that $(A,\succ,\prec,\Delta_{\succ,r},\Delta_{\prec,r})$
with $\Delta_{\succ,r},\Delta_{\prec,r}$ given by Eqs. (\ref{CD1})-(\ref{CD2}) is a quasi-triangular anti-dendriform bialgebra.
\item $T_r$ is a relative Rota-Baxter operator of weight $-1$ on $(A,\succ,\prec)$
 associated with the $A$-anti-dendriform algebra 
 $(A^{*},\succeq_{r+\tau(r)},\preceq_{r+\tau(r)},-R_{\cdot}^{*},L_{\prec}^{*},R_{\succ}^{*},-L_{\cdot}^{*})$,
 that is,
 \begin{align}\label{AD5} &
 T_{r}(\zeta)\succ T_{r}(\eta)= T_{r}(-R_{\cdot}^{*}(T_{r}(\zeta))\eta
 +L_{\prec}^{*}(T_{r}(\eta))\zeta-\zeta\succeq_{r+\tau(r)}\eta),\\&
 \label{AD6}T_{r}(\zeta)\prec T_{r}(\eta)=T_{r}(R_{\succ}^{*}(T_{r}(\zeta))\eta-L_{\cdot}^{*}(T_{r}(\eta))\zeta-\zeta\preceq_{r+\tau(r)}\eta).
 \end{align}
\item $T_r$ is a relative Rota-Baxter operator of weight $-1$ on $(A,\cdot)$
 associated with the $A$-associative algebra $(A^{*},\cdot_{r+\tau(r)},-R_{\prec}^{*},-L_{\succ}^{*})$, that is,
 \begin{align}\label{AD8} &
 T_{r}(\zeta)\cdot T_{r}(\eta)= T_{r}(-R_{\prec}^{*}(T_{r}(\zeta))\eta
 -L_{\succ}^{*}(T_{r}(\eta))\zeta-\zeta\cdot_{r+\tau(r)}\eta),
 \end{align}
  \end{enumerate}
 for all $\zeta,\eta\in A^{*}$, where $\succeq_{r+\tau(r)},\preceq_{r+\tau(r)}$ and $\cdot_{r+\tau(r)}$ are given respectively by 
 Eqs.~ (\ref{AD1})-(\ref{AD3}).
 
  \end{thm}
 \begin{proof} 
 In the light of Proposition \ref{Ya} and Theorem \ref{Ya1}, if $r+\tau(r)$ is invariant, then
 $r$ is a solution of the AD-YBE in $(A,\succ,\prec)$ if and only if
 \begin{equation*}T_{r}(\eta)\cdot T_{r}(\zeta)=T_{r}(L_{\succ}^{*}(T_{\tau(r)}(\zeta))\eta
-R_{\prec}^{*}(T_{r}(\eta))\zeta),~~\forall~\eta,\zeta\in A^{*}.
\end{equation*}
 Note that
  \begin{align*}T_{r}(\eta)\cdot T_{r}(\zeta)&=T_{r}(-L_{\succ}^{*}(T_{r}(\zeta))\eta
-R_{\prec}^{*}(T_{r}(\eta))\zeta+L_{\succ}^{*}(T_{r+\tau(r)}(\zeta))\eta)\\&=
T_{r}(-L_{\succ}^{*}(T_{r}(\zeta))\eta
-R_{\prec}^{*}(T_{r}(\eta))\zeta-\eta\cdot_{r+\tau(r)}\zeta).
\end{align*}
Thus, Item (a) $\Longleftrightarrow$ Item (c).
By the same token, Item (a) $\Longleftrightarrow$ Item (b). 
  \end{proof}
 
 If $r\in A\otimes A$ is skew-symmetric, then $r+\tau(r)=0$. The following conclusion is obtained.
 
 \begin{cor} \cite{030}
Let $(A,\succ,\prec)$ be an anti-dendriform algebra and $r\in A\otimes A$ be skew-symmetric. 
Then the following conditions are equivalent:
\begin{enumerate}
\item $r$ is a solution of the AD-YBE in $(A,\succ,\prec)$.
\item $T_r:A^{*}\longrightarrow A$ is a relative Rota-Baxter operator on $(A,\cdot)$
 associated to the representation $(A^{*},-R_{\prec}^{*},-L_{\succ}^{*})$.
 \item $T_r:A^{*}\longrightarrow A$ is a relative Rota-Baxter operator on $(A,\succ,\prec)$
  associated to the representation $(A^{*},-R_{\cdot}^{*},L_{\prec}^{*},R_{\succ}^{*},-L_{\cdot}^{*})$.
  \end{enumerate}
  \end{cor}

\section{Characterization of factorizable anti-dendriform bialgebras}
In this section, we first introduce the notion of quadratic Rota-Baxter anti-dendriform algebras. 
Then we characterize the relationship between factorizable anti-dendriform bialgebras
and quadratic Rota-Baxter anti-dendriform algebras.

A bilinear form $\omega$ on an associative algebra $(A,\cdot)$ is invariant if
$\omega(x\cdot y,z)=\omega(x,y\cdot z)$ for all $x,y,z\in A.$

A symmetric bilinear form $\omega$ on an associative algebra $(A,\cdot)$ 
is called a commutative Connes cocycle \cite{15} if the following equation holds:
\begin{equation*}\omega(x\cdot y,z)+\omega(y\cdot z,x)+\omega(z\cdot x,y)=0,~~\forall~x,y,z\in A.\end{equation*}

\begin{defi} Let $(A,\succ,\prec)$ be an anti-dendriform algebra 
and $\omega$ a non-degenerate symmetric bilinear form.
If $\omega$ is invariant, that is,
\begin{equation} \label{C1}\omega (x\succ y,z)=-\omega(y,z\cdot x), \ \ \
\omega (x\prec y,z)=-\omega(x,y\cdot z), ~\forall~x, y,z \in A.\end{equation}
Then $(A,\succ,\prec,\omega)$ is called a \textbf{quadratic anti-dendriform algebra}.
\end{defi}

\begin{thm}\label{Am3} \cite{15} Let $(A,\cdot)$ be an associative algebra
with a commutative Connes cocycle $\omega$.
 Then there exists a compatible anti-dendriform algebra structure $(A,\succ,\prec)$
 on $(A,\cdot)$ defined by Eq.~(\ref{C1}),
such that $(A,\cdot)$ is the associated associative algebra of $(A,\succ,\prec)$. This anti-dendriform algebra is called 
the compatible anti-dendriform algebra of $(A,\cdot,\omega)$. Moreover, $(A,\succ,\prec,\omega)$ is a quadratic anti-dendriform algebra.
Conversely, assume that $(A,\succ,\prec,\omega)$ is a quadratic anti-dendriform algebra.
Then $(A,\cdot)$ is an associative algebra
with a commutative Connes cocycle $\omega$.
\end{thm}

\begin{defi} Assume that $(A,\succ,\prec,P)$ is a Rota-Baxter anti-dendriform algebra of weight $\lambda$
and $(A,\succ,\prec,\omega)$ is a quadratic anti-dendriform algebra. Then $(A,\succ,\prec,P,\omega)$
is called a \textbf{quadratic Rota-Baxter anti-dendriform algebra of weight $\lambda$} if the following condition holds:
\begin{equation} \label{Fs}\omega (P(x),y)+\omega(x, P(y))+\lambda\omega(x,y)=0, ~\forall~x, y \in A.\end{equation}
\end{defi}

\begin{defi} Let $(A,\cdot,P)$ be a Rota-Baxter associative algebra of weight $\lambda$ and
 $\omega $ be a  commutative Connes cocycle. Then $(A,\cdot,P,\omega)$ is called a \textbf{Rota-Baxter associative
algebra of weight $\lambda$ with a commutative Connes cocycle} if the following condition holds: 
 \begin{equation} \label{Fs1}\omega (P(x),y)+\omega(x, P(y))+\lambda\omega(x,y)=0, ~\forall~x, y \in A.\end{equation}
\end{defi}

The following propositions can be formally derived from the equation:
\begin{align*} &\lambda\omega(x,y)+\omega (-\lambda(x)- P(x),y)+\omega(x, -\lambda(y)- P(y))\\=&
-\lambda\omega(x,y)-\omega (P(x),y)-\omega(x,  P(y)), ~\forall~x, y \in A.\end{align*}
 \begin{pro} \label{Fb2} 
Let $(A,\succ,\prec,\omega)$
 be a quadratic anti-dendriform algebra and let $P : A\longrightarrow A$ be a linear
map. Then $(A,\succ,\prec,P,\omega)$ is a
 quadratic Rota-Baxter anti-dendriform algebra of weight $\lambda$ if and only if $(A,\succ,\prec,-P-\lambda I,\omega)$
 is a quadratic Rota-Baxter anti-dendriform algebra of weight $\lambda$.
\end{pro}

\begin{pro} 
Let $(A,\cdot,\omega)$
 be a quadratic anti-dendriform algebra and let $P: A\longrightarrow A$ be a linear
map. Then $(A,\cdot,P,\omega)$ is a Rota-Baxter associative
algebra of weight $\lambda$ with a commutative Connes cocycle $\omega$
if and only if $(A,\cdot,-P-\lambda I,\omega)$ is
a Rota-Baxter associative
algebra of weight $\lambda$ with a commutative Connes cocycle $\omega$.
\end{pro}

\begin{thm}\label{Fb0} Suppose that $(A,\cdot,P,\omega)$ is a Rota-Baxter associative algebra of 
weight $\lambda$ with a commutative Connes cocycle. Then $(A,\succ,\prec,P,\omega)$ is a quadratic Rota-Baxter anti-dendriform
algebra of weight $\lambda$, where $\succ,\prec$ are defined by Eq.~(\ref{C1}).
 On the other hand, let $(A,\succ,\prec,P,\omega)$ be a quadratic Rota-Baxter anti-dendriform
 algebra of weight $\lambda$.
Then $(A,\cdot,P,\omega)$ is a Rota-Baxter associative algebra of weight $\lambda$ with a commutative Connes cocycle, 
where $\cdot=\succ+\prec$.
\end{thm}

\begin{proof} Let $(A,\cdot,P,\omega)$ be a Rota-Baxter associative algebra of 
weight $\lambda$ with a Connes cocycle. By Theorem \ref{Am3}, $(A,\succ,\prec,\omega)$ is a quadratic anti-dendriform algebra.
Using Eqs. (\ref{C1}) and (\ref{Fs1}), for all $x,y,z\in A$, we have
\begin{align*}& \omega( P(x)\succ P(y)-P(P(x)\succ y+x\succ P(y)+\lambda x\succ y),z)
\\=&-\omega( P(y), z\cdot P(x))+
\omega(P(x)\succ y, P(z))+\omega(x\succ P(y), P(z))\\&+\lambda\omega(x\succ y, P(z))+
\lambda\omega(P(x)\succ y, z)
+\lambda\omega(x\succ P(y), z)
+\lambda^{2}\omega(x\succ y, z)
\\=&
-\omega( P(y), z\cdot P(x))-\omega( y, P(z)\cdot P(x))-\omega( P(y), P(z)\cdot x)
-\lambda\omega( y, P(z)\cdot x)
-\lambda\omega( y, z\cdot P(x))
\\&-\lambda\omega( P(y), z\cdot x)
-\lambda^{2}\omega( y, z\cdot x)
\\=&-\omega( y, P(z)\cdot P(x))-\omega( P(y),  P(z)\cdot x+z\cdot P(x))+\lambda z\cdot x)
-\lambda \omega( y,  P(z)\cdot x+z\cdot P(x)+\lambda z\cdot x)
\\=&-\omega( y, P(z)\cdot P(x))+\omega(y, P( P(z)\cdot x+z\cdot P(x)+\lambda z\cdot x))
\\=&0,\end{align*}
which implies that $P(x)\succ P(y)=P(P(x)\succ y+x\succ P(y)+\lambda x\succ y)$.
Analogously, $P(x)\prec P(y)=P(P(x)\prec y+x\prec P(y)+\lambda x\prec y)$.
Thus, $(A,\succ,\prec,P,\omega)$ is a quadratic Rota-Baxter anti-dendriform algebra of weight $\lambda$.
The other hand is apparently.
We complete the proof.
\end{proof}

Let $\omega$ be a non-degenerate bilinear form on a vector space $A$. Then there is an isomorphism
$\omega^{\sharp}:A\longrightarrow A^{*}$ given by
\begin{equation} \omega(x,y)=\langle\omega^{\sharp}(x),y \rangle,~~\forall~x,y\in A.\end{equation}
Define an element $r_{\omega}\in A\otimes A$ with $T_{r_{\omega}}=(\omega^{\sharp})^{-1}$,
that is, 
\begin{equation}\label{Nd1} \langle T_{r_{\omega}}(\zeta),\eta\rangle=\langle r_{\omega},\zeta \otimes\eta\rangle=
\langle (\omega^{\sharp})^{-1}(\zeta), \eta\rangle. \end{equation}

\begin{lem}\label{Fb1} Let $(A, \succ,\prec)$ be an anti-dendriform algebra and $\omega$ be a non-degenerate bilinear form on $A$. Then
$(A, \succ,\prec,\omega)$ is a quadratic anti-dendriform algebra if and only if the corresponding $r_{\omega}\in A\otimes A$ via
Eq.~(\ref{Nd1}) is symmetric and invariant.\end{lem}
\begin{proof}
It is obvious that $\omega$ is symmetric if and only if  $r_{\omega}$ is symmetric. 
For all $x,y\in A$, put $\omega^{\sharp}(x)=\zeta,\omega^{\sharp}(y)=\eta,
\omega^{\sharp}(z)=\theta$ with $\zeta,\eta,\theta\in A^{*}$, we have
\begin{align*} &\omega (x\succ y,z)+\omega(y,z\cdot x)=\omega(x\succ (\omega^{\sharp})^{-1}(\eta),z)+\omega(y,
(\omega^{\sharp})^{-1}(\theta)\cdot x)
\\=&\langle x\succ (\omega^{\sharp})^{-1}(\eta),\theta\rangle+\langle \eta,(\omega^{\sharp})^{-1}(\theta)\cdot x\rangle
\\=&\langle x ,R_{\succ}^{*} ((\omega^{\sharp})^{-1}\eta)\theta\rangle+\langle L_{\cdot}^{*}((\omega^{\sharp})^{-1}\theta)\eta,
x\rangle
\\=& \langle x ,R_{\succ}^{*} (T_{r_{\omega}}(\eta))\theta+L_{\cdot}^{*}(T_{r_{\omega}}(\theta))\eta\rangle,\\ &
\omega (x\prec y,z)+\omega(x,y\cdot z)=\omega (\omega^{\sharp})^{-1}(\zeta)\prec y,z)+\omega(x,y\cdot (\omega^{\sharp})^{-1}(\theta))
\\=&\langle \theta,(\omega^{\sharp})^{-1}(\zeta)\prec y\rangle+\langle \zeta, y\cdot (\omega^{\sharp})^{-1}(\theta)\rangle
\\=&\langle L_{\prec}^{*}((\omega^{\sharp})^{-1}\zeta)\theta, y\rangle+\langle R_{\cdot}^{*}( (\omega^{\sharp})^{-1}\theta)\zeta, y\rangle
\\=&\langle L_{\prec}^{*}(T_{r_{\omega}}(\zeta))\theta+R_{\cdot}^{*} (T_{r_{\omega}}(\theta))\zeta, y\rangle
.\end{align*}
Combining Lemma \ref{In1}, symmetric $r_{\omega}\in A\otimes A$ is invariant if and only if $\omega$ satisfies Eq.~(\ref{C1}).
The proof is completed. 
 \end{proof}

\begin{pro} \label{QF1} Let $(A, \succ,\prec,\omega)$ be a quadratic anti-dendriform algebra and $r\in A\otimes A$. 
Assume that $r+\tau(r)$ is invariant. Define a linear map 
\begin{equation}\label{Nd2}P:A\longrightarrow A,\ \ \ P(x)=T_{r}\omega^{\sharp}(x),~~\forall~x\in A.\end{equation}
Then $r$ is a solution of the AD-YBE in $(A, \succ,\prec)$ if and only if $P$ satisfies 
\begin{align}&\label{Nd3}P(x)\succ P(y)= P(P(x)\succ y+x\succ P(y)-x\succ T_{r+\tau(r)}\omega^{\sharp}(y)),
\\&\label{Nd4}P(x)\prec P(y)= P(P(x)\prec y+x\prec P(y)-x\prec T_{r+\tau(r)}\omega^{\sharp}(y)),~~\forall~x,y\in A.
 \end{align}
\end{pro}

\begin{proof} In the light of Lemma \ref{Fb1}, $r_{\omega}$ is symmetric and invariant.
 For all $x,y\in A$, put $\omega^{\sharp}(x)=\zeta,\omega^{\sharp}(y)=\eta$ with 
$\zeta,\eta\in A^{*}$, we get
\begin{align*}& P(x)\succ P(y)=T_{r}\omega^{\sharp}(x)\succ T_{r}\omega^{\sharp}(y)=T_{r}(\zeta)\succ T_{r}(\eta),\\
&P(P(x)\succ y)=T_{r}\omega^{\sharp}(T_{r}(\zeta)\succ T_{r_{\omega}}(\eta))=
-T_{r}\omega^{\sharp} T_{r_{\omega}}(R_{\cdot}^{*}(T_{r}(\zeta))\eta)=-T_{r}(R_{\cdot}^{*}(T_{r}(\zeta))\eta),
\\
&P(x\succ P(y))=T_{r}\omega^{\sharp}(T_{r_{\omega}}(\zeta)\succ T_{r}(\eta))=
T_{r}\omega^{\sharp}T_{r_{\omega}}(L_{\prec}^{*}(T_{r}(\eta))\zeta)
=T_{r}(L_{\prec}^{*}(T_{r}(\eta))\zeta),
\\
&P(x\succ T_{r+\tau(r)}\omega^{\sharp}(y))
=T_{r}\omega^{\sharp}(T_{r_{\omega}}(\zeta)\succ T_{r+\tau(r)}(\eta))
=T_{r}(L_{\prec}^{*}(T_{r+\tau(r)}(\eta))\zeta)
=T_{r}(\zeta\succ_{r+\tau(r)}\eta).\end{align*}
 Thus, Eq.~(\ref{Nd3}) holds if and only if Eq.~(\ref{AD5}) holds.
 Analogously, Eq.~(\ref{Nd4}) holds if and only if Eq.~(\ref{AD6}) holds. The proof is finished.
 \end{proof}
 
\begin{lem} \label{QF2} Let $A$ be a vector space and $\omega$ be a non-degenerate symmetric bilinear form. Let $r\in A\otimes A$,
$\lambda\in k$ and $P$ be given by Eq.~(\ref{Nd2}). Then $r$ satisfies
\begin{equation} \label{Nd5} r+\tau(r)=-\lambda r_{\omega} \end{equation}
if and only if $P$ satisfies Eq.~(\ref{Fs}).
\end{lem}
\begin{proof} For all $x,y\in A$, put $\omega^{\sharp}(x)=\zeta,\omega^{\sharp}(y)=\eta,~~\zeta,\eta\in A^{*}$.
\begin{align*}& \omega(P(x),y)=\omega(y,P(x))=\langle\omega^{\sharp}(y),T_{r}\omega^{\sharp}(x)\rangle
=\langle \eta,T_{r}(\zeta)\rangle
=\langle r,\zeta\otimes\eta\rangle,\\
&\omega(x,P(y))=\langle \omega^{\sharp}(x),T_{r}\omega^{\sharp}(y)\rangle
=\langle \zeta,T_{r}(\eta)\rangle
=\langle r,\eta\otimes\zeta\rangle
=\langle \tau(r),\zeta\otimes\eta\rangle,
\\
&\lambda\omega(x,y)=\lambda\langle \omega(y,x)
=\lambda\langle \omega^{\sharp}(y),(\omega^{\sharp})^{-1}\omega^{\sharp}(x)\rangle
=\lambda\langle r_{\omega},\zeta\otimes\eta\rangle
.\end{align*}
Thus, Eq.~(\ref{Nd5}) holds if and only if Eq.~(\ref{Fs}) holds.
 \end{proof}
 
\begin{cor} Let $(A, \succ,\prec,P,\omega)$ be a quadratic Rota–Baxter anti-dendriform algebra of weight 0. Then there is a triangular
anti-dendriform bialgebra $(A, \succ,\prec, \Delta_{\succ_r},\Delta_{\prec_r})$ with $\Delta_{\succ,r},\Delta_{\prec,r}$ given
 by Eqs.~(\ref{CD1})-(\ref{CD2}), where $r\in A\otimes A$ given by 
$T_{r}(\zeta)=P(\omega^{\sharp})^{-1}$ for all $\zeta\in A^{*}$.
\end{cor}

 \begin{proof} Note that $r+\tau(r)=-\lambda r_{\omega}=0$ and $r$ is skew-symmetric.
Combining Proposition \ref{QF1} and Lemma \ref{QF2}, we finish the proof.
  \end{proof}
 
\begin{thm} \label{Fb3} Let $(A, \succ,\prec, \Delta_{\succ_r},\Delta_{\prec_r})$ be a
 factorizable anti-dendriform bialgebra
with $r\in A\otimes A$. Then $(A, \succ,\prec,P,\omega)$
 is a quadratic Rota-Baxter anti-dendriform algebra of
  weight $\lambda$ with $P$ given by Eq.~(\ref{Nd2}) and $\omega$ given by
\begin{equation} \label{Nd6} \omega_{T}(x,y)=-\lambda\langle T_{r+\tau(r)}^{-1}(x),y \rangle,~~\forall~x,y\in A.\end{equation}
Conversely, let $(A, \succ,\prec,P,\omega)$ be a quadratic Rota-Baxter anti-dendriform algebra of weight $\lambda~(\lambda\neq 0)$.
Then there is a factorizable anti-dendriform bialgebra $(A, \succ,\prec,\Delta_{\succ,r},\Delta_{\prec,r})$
 with $\Delta_{\succ,r},\Delta_{\prec,r}$ defined by Eqs.~(\ref{CD1})-(\ref{CD2}), where $r\in A\otimes A$ is
given through the operator form $T_r=P(\omega^{\sharp})^{-1}$.
\end{thm}

\begin{proof}  
In view of $(A, \succ,\prec, \Delta_{\succ,r},\Delta_{\prec,r})$ being a factorizable 
anti-dendriform bialgebra, then $r+\tau(r)$ is invariant and $T_{r+\tau(r)}$ is a linear isomorphism. 
Combining Proposition \ref{QF1} and Lemma \ref{QF2}, we get that $(A, \succ,\prec,P,\omega)$
 is a quadratic Rota-Baxter anti-dendriform algebra of weight $\lambda$,
  where $\omega^{\sharp}=-\lambda T_{r+\tau(r)}^{-1}$. 
 On the other hand, suppose that $(A, \succ,\prec,P,\omega)$ is a quadratic
  Rota-Baxter anti-dendriform algebra of weight $\lambda~(\lambda\neq 0)$.
 By Lemma \ref{Fb1}, Lemma \ref{QF2} and Proposition \ref{QF1}, 
 $r+\tau(r)$ is invariant, $T_{r+\tau(r)}=-\lambda (\omega^{\sharp})^{-1}$ is a linear isomorphism
 and $r$ is a solution of the AD-YBE in $(A, \succ,\prec)$. Thus,
  $(A, \succ,\prec,\Delta_{\succ,r},\Delta_{\prec,r})$ is a factorizable anti-dendriform bialgebra.
\end{proof}

\begin{pro} Let $(A, \succ,\prec,P)$ be a Rota-Baxter anti-dendriform algebra of weight $\lambda$. Then 
$(A\ltimes A^{*},\omega,P-(P^{*}+\lambda I))$
is a quadratic Rota-Baxter anti-dendriform algebra of weight $\lambda$, where the bilinear form $\omega$ on $A\oplus A^{*}$
is given by
\begin{equation*} \omega(x+\zeta,y+\eta)=\langle x,\eta\rangle+\langle y,\zeta\rangle,~~\forall~x,y\in A,~\zeta,\eta\in A^{*}.\end{equation*}
Then $(A\ltimes A^{*}, \Delta_{\succ_r},\Delta_{\prec_r})$
is a factorizable anti-dendriform bialgebra
with $\Delta_{\succ_r},\Delta_{\prec_r}$ defined by Eqs.~(\ref{CD1})-(\ref{CD2}) with $r$ given by $T_r=P(\omega^{\sharp})^{-1}$.
  Explicitly, assume that $\{e_1,\cdot\cdot\cdot, e_n\}$
is a basis of $A$ and $\{e_{1}^{*},\cdot\cdot\cdot, e^{*}_n\}$
is the dual basis, where $r=\sum_{i}e_{i}^{*}\otimes P(e_{i})-(P+\lambda I)(e_{i})\otimes e_{i}^{*}$
 
\end{pro}

\begin{proof} By direct computations, $(A\ltimes A^{*},\omega,P-(P^{*}+\lambda I))$
is a quadratic Rota-Baxter anti-dendriform algebra of weight $\lambda$.
For all $x\in A$ and $\zeta\in A^{*}$, $\omega^{\sharp}(x+\zeta)=x+\zeta$. By Corollary \ref{Fb3}, we have a linear map
$T_{r}:A\oplus A^{*}\longrightarrow A\oplus A^{*}$ defined by
\begin{equation*} 
T_{r}(x+\zeta)=(P-(P^{*}+\lambda I))(\omega^{\sharp})^{-1}(x+\zeta)=P(x)-(P^{*}+\lambda I)(\zeta).\end{equation*}
Thus, 
\begin{align*}&
\sum_{i,j}\langle r,e_{i}\otimes e_{j}^{*}\rangle=\sum_{i,j}\langle T_{r}(e_{i}), e_{j}^{*}\rangle=\sum_{i,j}\langle P(e_{i}), e_{j}^{*}\rangle
\\&\sum_{i,j}\langle r,e_{i}^{*}\otimes e_{j}\rangle=\sum_{i,j}\langle T_{r}(e_{i}^{*}), e_{j}\rangle=-\sum_{i,j}\langle (P^{*}+\lambda I)(e_{i}^{*}), e_{j}\rangle=-\sum_{i,j}\langle e_{i}^{*}, (P+\lambda I)(e_{j})\rangle.\end{align*}
Thus, $r=\sum_{i}e_{i}^{*}\otimes P(e_{i})-(P+\lambda I)(e_{i})\otimes e_{i}^{*}$.
\end{proof}

According to Theorem \ref{Fb3}, quadratic Rota-Baxter anti-dendriform algebras of non-zero weight $\lambda$ 
are one to one correspondence to factorizable anti-dendriform bialgebras. Combining Theorem \ref{Fb0}, a
Rota-Baxter associative algebra of weight $\lambda$ with a commutative Connes cocycle 
corresponds to a factorizable anti-dendriform bialgebra.

In the light of Proposition \ref{Fb2} and Theorem \ref{Fb3}, 
if $(A,\succ,\prec,P,\omega)$ is a quadratic Rota-Baxter anti-dendriform algebra of non-zero weight corresponds to a 
factorizable anti-dendriform bialgebra, then the quadratic Rota-Baxter anti-dendriform algebra $(A,\succ,\prec,-\lambda I-P,\omega)$
of non-zero weight is also corresponds to a factorizable anti-dendriform bialgebra.
By Proposition \ref{Qf} and Theorem \ref{Fb3}, if $(A, \succ,\prec, \Delta_{\succ,r},\Delta_{\prec,r})$ is a factorizable 
anti-dendriform bialgebra, then the factorizable 
anti-dendriform bialgebra $(A, \succ,\prec, \Delta_{\succ,\tau(r)},\Delta_{\prec,\tau(r)})$ gives rise to a 
quadratic Rota-Baxter anti-dendriform algebra of non-zero weight $\lambda$.
In fact, we have

\begin{pro}
Let $(A, \succ,\prec, \Delta_{\succ,r},\Delta_{\prec,r})$ be a factorizable 
anti-dendriform bialgebra which corresponds to a quadratic Rota-Baxter anti-dendriform algebras of non-zero weight $\lambda$. 
Then the factorizable 
anti-dendriform bialgebra $(A, \succ,\prec, \Delta_{\succ,\tau(r)},\Delta_{\prec,\tau(r)})$ corresponds to the 
quadratic Rota-Baxter anti-dendriform algebra $(A,\succ,\prec,-\lambda I-P,\omega)$ of non-zero weight $\lambda$. 
\end{pro}

\begin{proof} 
By Proposition \ref{Qf} and Theorem \ref{Fb1}, the factorizable 
anti-dendriform bialgebra $(A, \succ,\prec, \Delta_{\succ,\tau(r)},\Delta_{\prec,\tau(r)})$ corresponds to a
quadratic Rota-Baxter anti-dendriform algebra $(A,\succ,\prec,P',\omega')$ of non-zero weight $\lambda$.
By Theorem \ref{Fb3},
\begin{equation*} 
\omega'(x,y)=-\lambda\langle T_{r+\tau(r)}^{-1}(x),y \rangle=\omega(x,y).\end{equation*}
Using Eqs.~(\ref{Nd2}) and (\ref{Nd6}), 
\begin{align*} 
P'(x)= T_{\tau(r)}{\omega^{'}}^{\sharp}(x)= T_{\tau(r)}\omega^{\sharp}(x)
&=-\lambda T_{\tau(r)}T_{r+\tau(r)}^{-1}(x)=\lambda (T_{r}-T_{r+\tau(r)})T_{r+\tau(r)}^{-1}(x)
\\&=\lambda T_{r}T_{r+\tau(r)}^{-1}(x)-\lambda (x)
=-T_{r}\omega^{\sharp}(x)-\lambda (x)=-P(x)-\lambda (x)
.\end{align*}
Thus, the factorizable 
anti-dendriform bialgebra $(A, \succ,\prec, \Delta_{\succ,\tau(r)},\Delta_{\prec,\tau(r)})$ generates a
quadratic Rota-Baxter anti-dendriform algebra $(A,\succ,\prec,-\lambda I-P,\omega)$ of non-zero weight $\lambda$.
Analogously, the converse part holds.
\end{proof}


\begin{center}{\textbf{Acknowledgments}}
\end{center}
This work was supported by the Natural Science
Foundation of Zhejiang Province of China (LY19A010001), the Science
and Technology Planning Project of Zhejiang Province
(2022C01118).

\begin{center} {\textbf{Statements and Declarations}}
\end{center}
 All datasets underlying the conclusions of the paper are available
to readers. No conflict of interest exits in the submission of this
manuscript.


\end {document}